\documentclass[12pt,oneside]{amsart}
\usepackage[english]{babel}
\usepackage{amssymb}
\usepackage{amsthm}
\usepackage{mathtools}

\usepackage[foot]{amsaddr}
\usepackage{amsmath}
\usepackage{a4wide}
\usepackage{nccmath}
\usepackage{esvect}
\usepackage{ulem}
\usepackage{enumitem}
\usepackage{hyperref}
\usepackage{cleveref}
\usepackage{subfiles}
\hypersetup{
    colorlinks,
    linkcolor=red,
    citecolor=black,
    urlcolor=blue
}
\usepackage{verbatim}
\usepackage[dvips]{graphicx}
\usepackage{t1enc}
\usepackage{color}
\usepackage[margin=0.5in]{geometry}
\usepackage{graphicx}
\usepackage{color,soul}
\usepackage[parfill]{parskip}   
\usepackage[numbers,sort&compress]{natbib}
\usepackage{tikz}
\usepackage{amsmath}
\usepackage{caption}
\usepackage{subcaption}
\usepackage{mathtools}
\usepackage[section]{placeins}

\usetikzlibrary{arrows,positioning}
\usetikzlibrary{graphs}
\usetikzlibrary{graphs.standard}

\makeatletter
\def\subsubsection{\@startsection{subsubsection}{3}%
  \z@{.5\linespacing\@plus.7\linespacing}{.1\linespacing}%
  {\normalfont\itshape}}
\makeatother

\newtheorem{thm}{Theorem}[section]

\newtheorem{lemma}[thm]{Lemma}

\newtheorem{conjecture}[thm]{Conjecture}

\newtheorem{proposition}[thm]{Proposition}
\newtheorem{prop}[thm]{Proposition}
\newtheorem{problem}[thm]{Problem}

\newtheorem{cor}[thm]{Corollary}
\newtheorem{corollary}[thm]{Corollary}

\newcommand{\ignore}[1]{}

\begin{document}

\author{Rebekah Herrman}
\address[Rebekah Herrman]{Department of Industrial and Systems Engineering, The University of Tennessee, Knoxville, TN}
\email[Rebekah Herrman]{rherrma2@tennessee.edu}

\author{Stephen G. Z. Smith}
\address[Stephen G. Z. Smith]{}
\email[Stephen G. Z. Smith]{stephengz.smith@gmail.com}

\title[On the length of L-Grundy sequences] 
{On the length of L-Grundy sequences}

\linespread{1.3}
\pagestyle{plain}

\begin{abstract}
An L- sequence of a graph $G $ is a sequence of distinct vertices $S = \{v_1, ... , v_k\}$ such 
that $N[v_i] \setminus \cup_{j=1}^{i-1} N(v_j) \neq \emptyset$. The length of 
the longest L-sequence is called the L-Grundy domination number, denoted $\gamma_{gr}^L(G)$. In this 
paper, we prove $\gamma_{gr}^L(G) \leq n(G) - \delta(G) + 1$, which was conjectured by Bre{\v{s}}ar, 
Gologranc, Henning, and Kos. We also prove some early results about characteristics of $n$-vertex graphs such $\gamma_{gr}^L(G) = n$, as well as bounds on the change in L-Grundy number under graph operations.
\end{abstract}
\maketitle 

\section{Introduction}

In the zero-forcing process on a graph $G$, we select a subset of vertices and color them black. These vertices are said to be infected. At discrete time steps, if an infected vertex has one uninfected neighbor, that neighbor becomes infected, e.g. $u$ infects $v$ if $N[u]\setminus Z = \{v\}$, where $Z$ is the set of infected vertices.  The minimum number of vertices needed to infect the entire graph is called the zero-forcing number of $G$, denoted $Z(G)$. Recently, zero-forcing has been used to study multiplicities of eigenvalues of graphs \cite{lin2021zero} and the minimum rank and maximum nullity of graphs \cite{work2008zero, barioli2010zero, huang2010minimum, edholm2012vertex}. Additionally, it has been shown to be closely related to power domination \cite{benson2018zero, bozeman2019restricted}. 

A vertex $v$ of a graph $G$ totally dominates another vertex $u$ if they are neighbors, i.e. if $uv \in E$. $N(v)$ is the open neighborhood of $v$, which is the set of neighbors of $v$ and $N[v]$ is the closed neighborhood of $v$, $N[v] = N(v) \cup v$. Let $S = (v_1, \ldots, v_k)$ be a sequence of vertices such that $N[v_i] \setminus \cup_{j=1}^{i-1}N[v_j] \neq \emptyset$ for each $i \in [k]$, and denote the unordered elements of $S$ as $\widehat{S}$. For a graph $G$, the length of the longest sequence $S$ of vertices of $G$ is known as the Grundy domination number, denoted by $\gamma_{gr}(G)$. In other words, $S$ is the longest sequence such that each vertex in the sequence dominates a new vertex of $G$. For any vertex $v \in S$, $v$ footprints any $u \in N[v_i] \setminus \cup_{j=1}^{i-1} N[v_j]$. The index, or location, of a vertex $v$ in the sequence $S$ is written as $i(v)$.


Grundy domination was introduced in \cite{brevsar2014dominating} and it was later shown that the zero-forcing number is related to the 
Grundy domination number \cite{brevsar2017grundy}. In the following years, the related concepts of total Grundy domination (t-Grundy), Z-Grundy domination, and L-Grundy domination have been introduced and studied extensively \cite{brevsar2016total, brevsar2017grundy, brevsar2020grundy, brevsar2021graphs, lin2019zero, nasini2020grundy, bell2021grundy, brevsar2014dominating, brevsar2016graphchanges}. 

In this paper, we will provide results on $L$-Grundy domination that mirror previous results in t-Grundy domination. For a graph $G$, a sequence of vertices $S$ is an $L$-Grundy sequence if $N[v_i] \setminus \cup_{j=1}^{i-1} N(v_j)\neq \emptyset$.  The L-Grundy dominating number of a graph $G$, denoted by $\gamma_{gr}^L(G)$, is the longest length $L$-Grundy sequence. A sequence of vertices $S \subset V(G)$ is a t-Grundy sequence if $N(v_i) \setminus \cup_{j=1}^{i-1} N(v_j)\neq \emptyset$. The length of the longest t-Grundy sequence for a graph $G$ is the t-Grundy domination number, denoted $\gamma_{gr}^t(G)$. A vertex in an L-sequence (t-sequence) is said to L-footprint (t-footprint) another vertex if it mirrors the above footprint definition with the appropriate open and closed neighborhoods. When there is no chance of confusion, we will drop the modifier and only use footprint. Obtaining bounds for the t-Grundy domination number of a graph $G$ from its minimum and maximum degree was studied in \cite{brevsar2016total}. In particular, they showed that if the minimum degree of a graph $G$ is $\delta$, then $\gamma_{gr}^t(G) \leq n - \delta + 1$. For the L-Grundy number of $G$, it was conjectured that the same bound holds in \cite{brevsar2020grundy}. Specifically,

\begin{conjecture}[\cite{brevsar2020grundy}]\label{conjecture}
If $G$ on $n$ vertices with minimum degree $\delta$, then $\gamma_{gr}^L(G) \leq n - \delta + 1$.
\end{conjecture}

In Section \ref{sec:conj}, we prove the conjecture true.

In \cite{brevsar2016total} it was shown that graphs $G$ satisfying $\gamma_{gr}^t(G) = 2$ are complete multipartite and that no graph has t-Grundy number $3$. It was also shown that all graphs $G$ satisfying $\gamma_{gr}^t(G) = |V(G)|$ have very specific structure. In Section \ref{sec:characterizations}, we prove an analogous characterization of graphs with L-Grundy number equal to two and present early results towards the classification of graphs $G$ satisfying $\gamma_{gr}^L(G) = |V(G)|$.

The effects of edge or vertex removal on the Grundy and t-Grundy number was studied in \cite{brevsar2016graphchanges,brevsar2018totaldomgraphoperations}. We provide analogous results for edge and vertex removal for the L-Grundy number in Section \ref{sec:graphoperations} and conclude the paper with open problems in Section \ref{sec:conclusion}.



\section{Proof of Conjecture \ref{conjecture}}\label{sec:conj}

For the t-Grundy dominating number of a graph $G$ with minimum degree $\delta$, Bre{\v{s}}ar, Henning, and Rall  showed that $\gamma_{gr}^t(G) \leq n - \delta +1$ \cite{brevsar2016total}. Clearly, since any dominating number of a graph $G$ can be at most the number of vertices, $\gamma_{gr}^t(G) \leq n$, it follows that if $\delta = 1$,  $\gamma_{gr}^t(G) \leq n $. However, despite not proving that the same inequality holds in the L-Grundy regime, Bre\v{s}ar, Gologranc, Henning, and Kos  \cite{brevsar2020grundy} proved that if $G$ is an $n$-vertex graph satisfying $\gamma_{gr}^L(G) = n$, then there exists a vertex $u$ such that $deg(u) \leq 1$. We now prove that Conjecture \ref{conjecture} is true.

\begin{thm}\label{thm:conj_delta}
If $G$ is a graph on $n$ vertices with minimum degree $\delta$, then $\gamma_\mathrm{gr}^L(G) \leq n - \delta + 1$.
\end{thm}

\begin{proof}
Let $G$ be an $n$-vertex graph with minimum degree $\delta$. Since every graph has an L-Grundy set at most the size of the vertex set, if $\delta = 1$ then $\gamma_{gr}^L(G) \leq n$. Furthermore, any graph $G$ satisfying $\gamma_{gr}^L(G) = n$ must have a vertex of degree $1$. Therefore, if $\delta = 2$, then $\gamma_{gr}^L(G) \leq n-1$. For the remainder of the proof, assume that $\delta \geq 3$.

Suppose by contradiction that 
$\gamma_\mathrm{gr}^L(G) \geq n - \delta + 2$. For convenience, define $t = n - \delta + 2$ and let 
$S = (v_1, \ldots, v_k)$ with $k \geq t$ be the maximal $L$-sequence of $G$. Define $\widehat{T} = V(G) \setminus \widehat{S}$. It follows that $|\widehat{T}| \leq \delta - 2$.

Since $v_k$ has degree at least $\delta$, and 
there are at most $\delta - 2$ vertices in $\widehat{T}$, $v_k$ must be adjacent to at least two 
vertices in $\widehat{S}$. Thus, $v_k \not\in N[v_k] \setminus \cup_{i=1}^{k-1} N(v_i)$, 
and some $v \in N(v_k)$ must be in $N[v_k] \setminus \cup_{i=1}^{k-1} N(v_i)$. This 
neighbor must be either in $S$ or $\widehat{T}$.

Next, we show $v$ has degree at most $\delta - 1$. 
 Note that $v$ cannot have a neighbor in $S$, else $v \notin N[v_k] \setminus \cup_{i=1}^{k-1} N(v_i)$, thus its only neighbors can be in $\widehat{T}$ or are $v_k$. If $v$ is adjacent to only vertices of $\widehat{T}$ and $v_k$, then it has degree at most $\delta -1 $, a contradiction.


Therefore, $\gamma_\mathrm{gr}^L(G) \leq n - \delta + 1$.

\end{proof}

\section{L-Grundy Domination Numbers of Graphs With a View Towards Characterization }\label{sec:characterizations}

\subsection{Graphs with small L-Grundy numbers and results for cliques}
\hfill\\
Related to the problem of characterizing graphs with L-Grundy number equal to the number of vertices is characterizing graphs with small L-Grundy numbers. Interestingly, it was shown in \cite{brevsar2018totaldomgraphoperations} that there is no graph $G$ such that $\gamma_{gr}^t(G) = 3$. In the L-Grundy regime however, that is untrue since a path on three vertices, $P_3$, proves otherwise. Moreover, $P_3$ is not the only graph with L-Grundy number equal to $3$ (delete an edge from a $K_4$). Additionally, it was shown in \cite{brevsar2016total} that a graph has t-Grundy number $2$ if and only if it is a complete multipartite graph. We prove that if $G$ is a connected graph, $\gamma_{gr}^L(G) = 2$ if and only if $G = K_n$.

\begin{lemma}\label{lem:neighbor_with_L_2} If G is a connected graph on $n$ vertices and $\gamma_\mathrm{gr}^L(G) = 2$, then the vertices in the $L$-set are adjacent.

\end{lemma}

\begin{proof}
Let $G$ be an $n$-vertex graph, $\gamma_\mathrm{gr}^L(G) = 2$, and $S = (x,y)$ be an $L$-sequence of length two. Suppose by contradiction that $x$ and $y$ are not adjacent.  If G has only two vertices, this contradicts the fact that $G$ is connected. Thus, suppose $G$ has at least three vertices. Since $G$ is connected, this implies there is at least one $v$ that is a neighbor of either $x$, $y$, or both. Then, $(x,y,v)$ is a L-sequence of length three, a contradiction.

Thus, the two vertices of the L-sequence are adjacent.
\end{proof}

With this lemma, we can now characterize graphs that have L-Grundy dominating number equal to two.

\begin{thm}\label{thm:completegraph}
If $G$ is a connected graph on $n \geq 3$ vertices, then $\gamma_\mathrm{gr}^L(G) = 2$ if and only if $G=K_n$.
\end{thm}

\begin{proof}
Since $\gamma_\mathrm{gr}^L(G) = 2$, the two vertices in the L-sequence, $x$ and $y$, are adjacent to each other. Any vertex $v \in V(G)$ is adjacent to both $x$ and $y$, else $(v,y,x)$ or $(v,x,y)$ is a longer L-sequence. Thus, the $K_3$ case is complete. 

Now consider $n \geq 4$, and let $x,y$ be the vertices in the length two L-sequence. For any pair of vertices $v_1, v_2 \notin \{x,y\}$, $v_1$ must be adjacent to $v_2$, else $\{x, v_1, y\}$ is a L-sequence. Thus, all vertices must be pairwise adjacent, so $G = K_n$.

For the reverse direction, suppose by contradiction that $G \neq K_n$.

Let $G$ be an $n$-vertex graph with L-sequence $S = (v_1, v_2)$. By Lemma \ref{lem:neighbor_with_L_2}, $v_1$ and $v_2$ are adjacent. If $N(v_1) = N(v_2)$, then $N(v_1)$ is not a clique by assumption. Let $x, y$ be the two non-adjacent vertices in $N(v_1)$. Then $S' = (x,y,v_1)$ is a strictly longer $L$-sequence because $x$ does not footprint $y$, $y$ is not in $N(x)$, so $y$ does not footprint $x$, and $v_1$ footprints both $x$ and $y$. This contradicts our assumption that $S$ was maximal.

If $N(v_1) \neq N(v_2)$, without loss of generality, let $u$ be a vertex in $N(v_1)$ but not in $N(v_2)$. Then $S' = (u,v_2,v_1)$ is a strictly longer $L$-sequence because $u$ footprints $v_1$ and not $v_2$, $v_2$ footprints itself since $u$ is not adjacent to $v_2$, and finally, $v_1$ footprints $v_2$. This contradicts out assumption that $S$ was maximal.

Therefore, if $\gamma_\mathrm{gr}^L(G) = 2$, $G = K_n$.

\end{proof}

It is worth noting that the empty graph on two vertices also has L-Grundy domination number equal to two. This result on complete graphs motivates the study of sets of vertices who all have the same closed neighborhood.



\begin{proposition}
Let $G$ be a graph. If $G$ contains three pairwise adjacent vertices $v_1, v_2, v_3$ such that $N(v_1) = N(v_2) = N(v_3)$, then $\gamma_{gr}^L(G) \neq n$.
\end{proposition}

\begin{proof}
For $i \neq j \neq m \in \{1,2,3\}$, $N[v_i] \setminus (N(v_j) \cup N(v_m)) = \{ \emptyset\}$ since $v_i \in N(v_j) \cup N(v_m)$, $v_j \in N(v_m)$, and $N(v_i)\setminus \{v_j\} = N(v_j)\setminus \{v_i\}$. Thus, $v_1$, $v_2$, and $v_3$ cannot all be included in an L-sequence, so $\gamma_{gr}^L(G) \neq n$.
\end{proof}

Thus, we know that if a graph on $n$ vertices contains a clique of size $k$, in order for the L-Grundy number to be equal to $n$, at most two of the vertices in the clique can have degree $k-1$. There exists graphs with cliques such that exactly two vertices have degree $k-1$. For example, consider $K_3$ with vertices $v_1, v_2, v_3$ and attach a single vertex, $v_4$ to $v_3$ as shown in Figure \ref{fig:cliqueexample}. The sequence $\{v_1, v_2, v_4, v_3\}$ is a $L$-Grundy sequence of length four.

\begin{figure}[!htbp]
 \centering
 \begin{tikzpicture}[scale=2]
\begin{scope}[every node/.style={scale=.75,circle,draw}]
    \node (A) at (0,0) {$v_1$};
    \node (B) at (.5,1) {$v_2$};
	\node (C) at (1,0) {$v_3$}; 
	\node (D) at (2,0) {$v_4$};

\end{scope}

\draw  (A) -- (B);
\draw  (A) -- (C);
\draw  (B) -- (C);
\draw  (D) -- (C);

\end{tikzpicture}
\caption{An example of a graph containing a clique $K_3$ where two vertices have degree $k-1 = 2$. This has an $L$-grundy sequence of length four.}
\label{fig:cliqueexample}
\end{figure}
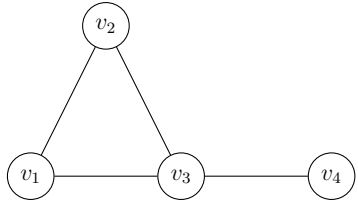

We can immediately define an infinite family of non-tree graphs whose L-Grundy domination number is equal to the number of vertices due to this result. 

\begin{corollary}\label{infinitefamily}
For all $n$, there exists a non-tree graph $G$, $|V(G)| = n$ such that $\gamma_{gr}^L(G) = n$.  
\end{corollary}

\begin{proof}

Consider the graph  $K_k$ with vertices $u_i$ for $i \in [k]$ and attach a degree one vertex, $v_i$, to every vertex except one, without loss of generality $u_k$. Then, the $L$-sequence $S = \{u_k, v_i, ... , v_k, u_1, ... , u_{k-1} \}$. Since $N[v_i] \setminus \cup_{j = 1}^{i-1} N(v_j) \cup u_k = \{v_i\}$ and for $i \neq k$, $N[u_i] \setminus (\cup_{j = 1}^{i-1}N(u_j) \cup_{j = 1}^{k}N(v_j) \cup N(u_k)) = \{v_i\}$.
 
 \end{proof}

\subsection{Cycles}
\hfill\\
Since all graphs except forests contain cycles, it is natural to determine the $L$-grundy number of a cycle $C_n$, and attempt a characterization from there.

\begin{lemma}\label{lem:cycles}
For every $n \geq 3$, $\gamma_{gr}^L(C_n) = n - 1$.
\end{lemma}
\begin{proof}
Let $C_n = v_1v_2 \cdots v_nv_1$ We consider two cases; $n$ even, and $n$ odd.

We construct a L-sequence $S$ as follows. The first $\lceil n/2 \rceil$ vertices to be added to $S$ are the vertices with odd indices in increasing order. Immediately following are the vertices with even indices in descending order excluding $v_2$.

If $n$ is even, $S_{n/2} = (v_1, v_3, \ldots, v_{n-1})$ is a L-grundy subsequence because, with the exception of $v_1$ and $v_{n-1}$, each vertex $v_i$, with an odd index $i$, footprints vertex $v_{i+1}$. In the case of $v_1$, it footprints both $v_2$ and $v_n$ while $v_{n-1}$ footprints itself. The remaining vertices form a L-sequence because, excluding $v_n$, each vertex $v_i$, with even index $i$, footprints $v_{i-1}$. $v_n$ footprints both $v_1$ and $v_{n-1}$.

The case for $n$ odd is identical. $S_{\lceil n/2 \rceil}$ consists of all vertices with odd indices, while the remaining are those with even indices in descending order, except $v_2$. The only difference is that no vertices are required to footprint themselves. Every vertex $v_i$, with an odd index $i$, footprints $v_{i+1}$ except $v_1$ and $v_n$. $v_1$ footprints both $v_2$ and $v_n$, while $v_n$ only footprints $v_1$ since its even neighbor was footprinted by $v_{n-2}$. Finally, every vertex $v_i$ with even index $i$ footprints $v_{i-1}$.
\end{proof}

We can easily obtain a graph $G'$ whose L-Grundy domination number is equal to the number of vertices by adding a single vertex and edge to $C_n$.

\begin{thm}\label{thm:cyclewithleaf}
Let $G$ be an $n$-vertex graph obtained by attaching a leaf to a cycle. Then $\gamma_{gr}^L(G) = n$
\end{thm}

\begin{proof}
Let $C_k = v_1v_2\cdots v_kv_1$ be a cycle on $k$ vertices, and let $u$ be a leaf attached to an arbitrary vertex $C_k$. By relabeling the vertices of the cycle if necessary, we may assume $u$ is adjacent to $v_2$. From Lemma \ref{lem:cycles}, we have a L-sequence $S$ of length $k-1$. We obtain a new L-sequence $S'$ from $S$ by first prepending $S$ with $u$, then adding $v_2$ to the end. $S'$ is a L-sequence because, even though $v_1$ no longer footprints $v_2$, it still footprints $v_n$, $u$ footprints $v_2$, and $v_2$ is the only vertex that can footprint $u$, aside from $u$.
\end{proof}

\par

\subsection{Families of graphs }
\hfill\\
Let $G$ be an $n$ vertex graph such that $\gamma_{gr}^L(G) < n$. We now give a method of obtaining graphs $G'$ from $G$ such that $\gamma_\mathrm{gr}^L(G') = n'$, where $n'$ is the number of vertices of $G'$.

\begin{thm}\label{addingmultiplevertices}
Let $G$, $|V(G)| = n$. If $\gamma_{gr}^L(G) = k < n$, there exists a graph $G'$, $|V(G')| = 2n-k$ where $\gamma_{gr}^L(G') = 2n-k$
\end{thm}

\begin{proof}
Let $G$ be a graph on $n$ vertices with $\gamma_{gr}^L(G) = k < n$ where $S$ is its maximal L-Grundy sequence. Let $G'$ be a graph that is equal to $G$ with a single vertex, $v_i$ connected to each vertex not in $S$, $u_i$. Then, $S$ is still a $L$-grundy sequence since the neighborhoods of all vertices in $\widehat{S}$ remained unchanged. Let $S'$ be obtained from $S$ by setting the first $k$ elements of $S'$ to the whole of $S$, followed by all leaves $v_i$ in any order, followed by all vertices in $G \setminus \widehat{S}$. This sequence has length $2n-k$ since there are $n$ vertices in $G$ and $n-k$ added vertices. It remains to show that $S'$ is a $L$-grundy sequence.

As stated earlier, $S$ is a L-sequence. Let us denote the vertices of $S$ as $\{x_i\}_{i \in [k]}$ for some $k \in \mathbf{N}$. Now for any $m < n-k$, $v_m \in N[v_m] \setminus (\cup_i N(x_i) \cup_{j=1}^{m-1} N(v_j)) $  since the only neighbor of $v_m$ is a vertex in $G \setminus \widehat{S}$. Finally, for $m < n-k$, $v_m \in N[u_m] \setminus (\cup_i N(x_i) \cup_{j=1}^{n-k} N(v_j) \cup_{l=1}^{n-k} N(u_l)) $ since the only neighbor of $v_m$ is $u_m$. 
\end{proof}

Recall that $S$ is a t-Grundy sequence if $N(v_i) \setminus \cup_{j=1}^{i-1}N(v_j) \neq \emptyset$, for all $i \in [k] \setminus \{1\}$. The following characterization of graphs with t-Grundy dominating number equal to the number of vertices was proved by Bre{\v{s}}ar, Henning, and Rall in \cite{brevsar2016total}.

\begin{prop}\label{prop:tcharacter}[\cite{brevsar2016total}]
Let $G$ be an $n$-vertex graph. $\gamma_{gr}^t(G) = n$ if and only if there exists an integer $k$ such that $n = 2k$ and the vertices can be labeled $\{x_1, \ldots, x_k, y_1, \ldots, y_k\}$ such that the following holds. For each $i \in [k]$, $x_i$ is adjacent to $y_i$, the set of vertices $\{x_1, \ldots, x_k\}$ is independent, and $y_j$ adjacent to $x_j$ means $i \geq j$.
\end{prop}

If $G$ is an $n$-vertex graph with $\gamma_{gr}^t(G) = n$, then $\gamma_{gr}^L(G) = n$. If $G$ does not satisfy the conditions of Proposition \ref{prop:tcharacter}, then a new graph $G'$ can be obtained from $G$ by adding a leaf to every vertex of degree $2$ or higher such that $\gamma_{gr}^L(G')  = n'$ where $n'$ is the number of vertices in $G'$. The following theorem shows that full L-Grundy number can be achieved by attaching leaves to every vertex with degree $3$ or higher.

\begin{thm}\label{thm:leaves}
Let $G$ be an $n$-vertex graph with at least one vertex of degree three or greater that satisfies the condition that every vertex $v \in V(G)$ such that $\deg(v) \geq 3$ has a neighbor of degree one. Then $\gamma_{gr}^L (G) = n$. 
\end{thm}

\begin{proof}
We begin with an empty L-sequence $S$ and add vertices to it in the following order: vertices of degree $2$, vertices of degree $1$, and vertices of degree $3$ or higher.

First, consider a vertex $v$ of degree $2$ such that each neighbor has degree $3$ or higher. Since higher degree vertices have not been added yet, $v$ footprints such neighbors unless another degree $2$ vertex has footprinted them previously, in which case, $v$ footprints itself. Therefore, if $v$ has degree $2$, and both of its neighbors have higher degrees, then we add them in any order to $S$. Next, suppose that $u$ and $v$ are two adjacent vertices with degree $2$, and that each has a single neighbor of degree $3$ or higher. If both of those higher degree neighbors have not been previously footprinted, then it is clear that $u$ and $v$ are admissible to $S$. If their neighbors have already been footprinted, then it is also clear that they may be added as $u$ would footprint $v$ and vice versa. Finally, suppose that $v_1v_2 \cdots v_n$ is a path (of degree $2$ vertices) of length $n \geq 3$ such that the endpoints $v_1$ and $v_n$ both have a single neighbor of higher degree. Then, we add these vertices in the obvious order, stopped at $v_{n-2}$. $v_1$ footprints $v_2$, $v_2$ footprints $v_3$, and so on. If the higher degree neighbor of $v_n$ has not been footprinted, then we can add $v_{n-1}$ and $v_n$ in that order. If it has been footprinted, then we add $v_n$ before $v_{n-1}$ since $v_n$ footprints itself, and finishing by adding $v_{n-1}$, which footprints $v_n$. This completes the addition of all degree $2$ vertices to $S$.

Next, it is clear that we may add all degree $1$ vertices because they footprint themselves if their higher degree neighbors have been previously footprint, or they footprint themselves and their higher degree neighbors. And to finish, since every vertex of degree $3$ or higher has a neighbor with degree $1$, it must be added to $S$ since it is the only vertex that can footprint the degree $1$ vertices.
\end{proof}

This result implies that given any graph, $G$, we can derive a graph  from it $G'$ such that $\gamma_{gr}^L (G') = n$. This condition is not necessary for $\gamma_{gr}^L(G) = n$, however, as it was shown in \cite{brevsar2020grundy}, any forest on $n$ vertices has an $L$-Grundy number of size $n$, and there exists forests that do not satisfy this condition, such as a collection of paths of length at least three. 

\section{Graph operations}\label{sec:graphoperations}

\subsection{Edge Removal}
\hfill\\
In  \cite{brevsar2016graphchanges} it was shown how the Grundy dominating number changes under the operation of edge removal. In particular, they showed that under single edge removal, the Grundy domination number can increase or decrease by at most one. We show that the change in the L-Grundy number differs slightly due to the fact that vertices may footprint themselves. Indeed, we show that the L-Grundy domination number increases by at most two or decrease by at most one.

\begin{thm}\label{thm:remove_edge}
If $G$ is a graph and $e \in E(G)$, then $\gamma_{gr}^L(G) - 1 \leq \gamma_{gr}^L(G - e) \leq \gamma_{gr}^L(G) + 2$.
\end{thm}
\begin{proof}
Let $e = uv \in E(G)$ and set $G - uv = G'$. For the lower bound, let $S = \{v_1, ... , v_k\}$ be a L-sequence of $G$. Let us consider $G'$. If $u, v \notin \widehat{S}$, then $S$ is an L- sequence of $G'$. Without loss of generality, let $u$ be in $\widehat{S}$. Removing $u$ from $S$ still gives an L-sequence of $G'$ since removing $u$ from $\widehat{S}$ gives an L-sequence of $G$. This is because removing a vertex $v_i \in \widehat{S}$ does not increase the size of $\cup_{j \in J} N(v_j)$ for $i \in J$.

Without loss of generality, suppose $u$ footprints $v$ and $v$ footprints $y \neq u$. In $G - e$, $u$ can no longer footprint $v$, so if it does not footprint itself or another vertex, it is removed from the L-sequence. If $v$ does not have any other neighbor in $\widehat{S}$ other than $u$, then we add the appropriate vertex that footprints it. Thus, $|S'|$ is equal to $|S|$, $|S \setminus \{u\}|$, $|S \cup \{v'\}|$, or $|(S \setminus \{u\}) \cup \{v'\} |$, where $v'$ is the footprinter of $v$.

The only possibility in which we may need to remove two vertices from the L-sequence is if $u,v \in \widehat{S}$ and we have that $u$ footprints $v$ and $v$ footprints $u$. If either $u$ or $v$ footprinted other vertices as well, then they would be left in the L-set, so we assume that $u$ footprints only $v$ and $v$ footprints only $u$. Without loss of generality, we may assume that $i(u) < i(v)$ and, since $v$ footprints $u$, $u$ footprints itself. Then, in $G-e$, $u$ still footprints itself, therefore $u$ must remain in the new L-sequence. Indeed, let $S'$ be obtained from $S$ by removing $v$, thus, $|S'|$ is equal to $|S|-1$, $|S|$, or $|S| + 1$.

To see the upper bound, if $S'$ is an L-sequence of $G-e$, then $S := S' \setminus \{u,v\}$ is an L-sequence of $G$ and the result follows. Note that since a single edge is removed, all vertices not incident to edge and that were in $S'$ can still be included in $S$
\end{proof}

\begin{cor}\label{cor:k-edge-removal}
Let $G$ be a graph and let $G'$ be obtained from $G$ by adding $k$ edges to $G$. Then $\gamma_{gr}^L(G) - k \leq \gamma_{gr}^L(G') \leq \gamma_{gr}^L(G) + 2k$.
\end{cor}

To illustrate these changes in the L-Grundy number, we now provide an example of a graph achieving each value in the theorem's bounds. Consider first the increase in two of the L-Grundy number of a graph $G = K_4 \, \Box \, K_2$. We first note that $\gamma_{gr}^L(K_4  \, \Box \, K_2) = 4$. To see this, first label the vertices of the $K_4$'s with $\{v_1, v_2, v_3, v_4\}$ for the first, and $\{u_1, u_2, u_3, u_4\}$ for the second, and recall that $v_i$ is adjacent to $u_i$. Suppose by contradiction that $\gamma_{gr}^L(K_4  \, \Box \, K_2) \geq 5$. Then there is at least one element from each $K_4$ in $S$, i.e. at least one $u_i$ and at least one $v_j$, and, without loss of generality, at least three elements in $\{u_1, u_2, u_3, u_4\}$ are in $S$. If exactly three $u_i$ are in $S$, then there are at least two elements from $\{v_1, v_2, v_3, v_4\}$. Thus $u_i$ and $v_i$ and $u_j$ and $v_j$ in $S$. Notice, however, that $u_i$ and $v_i$ footprint everything, so the L-sequence is maximal when they both are in $S$. The same is true for $u_j$ and $v_j$. Thus, without loss of generality, place all $u_i$ in $S$ first. This implies everything in $G$ is footprinted except for one $v_k$. Once either $v_i$ or $v_j$ is placed in $S$, $v_k$ will be footprinted, so the other cannot be in $S$, as it has nothing to footprint. The proof for all four $u_i$ being included in $S$ is the same, except now no $v_i$ can be in $S$. Thus, $\gamma_{gr}^L(K_4  \, \Box \, K_2) < 5$. In fact, $\gamma_{gr}^L(K_4  \, \Box \, K_2) = 4$ since all $u_i$ can be in $S$. To see the L-Grundy domination number increase by two after removing an edge, let $S = \{u_1, u_2, u_3, u_4\}$. Suppose edge $u_1v_1$ is removed. Then $u_1$ no longer footprints $v_1$, so by appending $v_1$ to the end of $S$, followed by $v_i$ for $i > 1$, we have created a longer L-sequence by two.

Consider the L-Grundy number of the complete graph, $\gamma_{gr}^L(K_n) = 2$. Let $u$ and $v$ be the two vertices in $\widehat{S}$, according to Theorem \ref{thm:completegraph}. Removal of $e = uv$ from $K_n$ now allows us to add any other vertex to the end of $S$. If $x$ is such a vertex, then our new L-Grundy sequence is $(u,v,x)$ as $u$ footprints every vertex except $v$, $v$ footprints itself, and $x$ footprints both $u$ and $v$.

Next, consider a cycle with a leaf attached,  as in Theorem \ref{thm:cyclewithleaf}. If $u$ is the leaf, then removal of any edge not incident with $u$ results in a tree, whose L-Grundy number is equal to the number of vertices. Since the cycle with leaf has L-Grundy number equal to the number of vertices as well, it is clear that the L-Grundy number remains the same.

Finally, to see a drop of one, consider the graph $G$ obtained by taking two copies of a cycle $C_n$, $n \geq 3$, and adding a single edge between one vertex in each cycle. Clearly, $\gamma_{gr}^L(G) = 2n-1$. By then removing the single edge between the cycles, our disconnected graph now has L-Grundy number $2n-2$, as each cycle on $n$ vertices has L-Grundy number $n-1$.

\subsection{Vertex removal}
\hfill\\
 Let $H$ be an induced subgraph of $G$. If $S$ is an L-sequence for $H$, then it is also a L-sequence for $G$. Therefore $\gamma_{gr}^L(H) \leq \gamma_{gr}^L(G)$, and the L-Grundy dominating number is bounded above due to hereditary graph properties. In \cite{brevsar2016graphchanges}, it was shown that for the Grundy dominating number, $\gamma_{gr}(G) - 2 \leq \gamma_{gr}(G-u) \leq \gamma_{gr}(G)$. Similarly, it was shown in \cite{brevsar2018totaldomgraphoperations} that for the t-Grundy regime, vertex deletion affects the bounds by an identical amount. That is $\gamma_{gr}^t(G) - 2 \leq \gamma_{gr}^t(G-u) \leq \gamma_{gr}^t(G)$. Since the L-Grundy number differed from the Grundy number under the graph operation of edge deletion, it is natural to assume there may be some difference. In the following theorem, we show that despite this ability of vertices to footprint themselves, the bounds remain the same for the operation of vertex deletion.

\begin{thm}\label{thm:vertex_removal}
If $G$ is a graph and $u \in V(G)$, then $\gamma_{gr}^L(G) - 2 \leq \gamma_{gr}^L(G - u) \leq \gamma_{gr}^L(G).$
\end{thm}
\begin{proof}
Let $S$ be an L-Grundy dominating sequence of $G$, and let $u \in V(G)$ be the deleted vertex such that $G' = G - u$. Since $G'$ is an induced subgraph of $G$, we must have $\gamma_{gr}^L(G') \leq \gamma_{gr}^L(G)$. For the lower bound, suppose that $u, v \in \widehat{S}$ and that both $v$ footprints $u$ and $u$ footprints $v$. When removing a vertex $u$ from the L-sequence, we must also consider the effect the removal of $u$ has on the vertex that footprints $u$. Since a vertex may be footprinted by at most two vertices, if $v$ does not footprint itself or any vertices other than $u$, a new L-sequence $S'$ is obtained from $S$ by removing $u$ and $v$. Thus, $|S'| = |S \setminus \{u,v\}| = |S|-2$.
\end{proof} 

We now provide a few examples of graphs satisfying the range of values in the above proof, showing the bounds in the theorem are tight.

Consider the L-Grundy number of a complete graph, $K_n$, $n \geq 3$. Clearly, $\gamma_{gr}^L(K_n) = 2$, and removing any vertex, we have a $K_{n-1}$, whose L-Grundy number remains $2$. To achieve $\gamma_{gr}^L(G-u) = \gamma_{gr}^L(G) - 1$, let $G$ be a path $P_n$, $n \geq 2$. Let $u$ be an end vertex of the path and observe that $\gamma_{gr}^L(P_n) = n$. Then, $\gamma_{gr}^L(P_n - u) = n-1$. Finally, to see an L-Grundy number drop of $2$, consider the graph $G$ which is $C_n$ with a leaf attached, as in Theorem \ref{thm:cyclewithleaf}. Then $\gamma_{gr}^L(G) = n+1$. By removing the leaf from $G$, we now have a cycle with $n$ vertices whose L-Grundy number is $n-1$.

\section{Conclusion}\label{sec:conclusion}
We believe the problem of characterizing $n$-vertex graphs $G$ satisfying $\gamma_{gr}^L(G) = n$ to be difficult. In this paper, we have determined sufficient conditions for leaf placement in $G$, but determining the necessary conditions proves to be more challenging.  In \cite{brevsar2020grundy}, it was shown that every $n$-vertex tree $T$ satisfies $\gamma_{gr}^L(T) = n$ and that if $G$ is an  $n$-vertex graph such that $\gamma_{gr}^L(G) = n$, then $\delta = 1$. We have proved that if $G$ is a graph such that $\gamma_{gr}^L(G) = k < n$, then by obtaining $G'$ from $G$ by attaching a leaf to each vertex of degree $3$ or higher, we have $\gamma_{gr}^L(G') = n'$, where $n'$ is the number of vertices in $G'$. Furthermore, we were also able to construct an infinite family of graphs with only one degree one vertex such that $\gamma_{gr}^L(G) = |V(G)|$. In light of this, we restate the following problem, originally proposed in \cite{brevsar2017grundy}.

\begin{problem}[\cite{brevsar2017grundy}]
Characterize graphs $G$ with $\gamma_{gr}^L(G) = \mid V(G) \mid$.
\end{problem}

In Section \ref{sec:graphoperations}, we showed that the removal of any one edge resulting in the L-Grundy number increasing by $2$ or decreasing by $1$. Clearly, if $\gamma_{gr}^L(G) \geq 3$, the number of edges needed to make $G$ a clique cannot be added since a clique has L-Grundy number $2$, but it would be interesting to determine which sorts of structures permit stability in the L-Grundy number provided edges are added or removed.

In \cite{brevsar2016graphchanges}, it was shown that if $G$ is a graph and $e \in E(G)$, then $\gamma_{gr}(G) - 1 \leq \gamma_{gr}(G-e) \leq \gamma_{gr}(G)+1$. Moreover, there exists a specific graph $G$ that realizes every value between $\gamma_{gr}(G) - 1$ and $\gamma_{gr}(G)+1$ for different edges in $G$. Similarly, if $u \in V(G)$, there exists a specific graph $G$ that realizes every value between $\gamma_{gr}(G)-2$ and $\gamma_{gr}(G)$. We were unable to find such a graph that realizes every value in the bounds of Theorem \ref{thm:remove_edge} or Theorem \ref{thm:vertex_removal}. This motivates the following two problems.

\begin{problem}
Find a graph $G$ such that all the values of $\gamma_{gr}^L(G-e_j)$ are realized between $\gamma_{gr}^L(G)-1$ and $\gamma_{gr}^L(G)+2$ for some edges $e_j \in E(G)$, or prove that one does not exist.
\end{problem}

\begin{problem}
Find a graph $G$ such that all the values of $\gamma_{gr}^L(G-u_i)$ are realized between $\gamma_{gr}^L(G)-2$ and $\gamma_{gr}^L(G)$ for some vertices $u_i \in E(G)$,  or  prove that one does not exist.
\end{problem}

\section*{Acknowledgements}

Both authors wish to thank to Tomas Ju{\v{s}}kevi{\v{c}}ius and Peter van Hintum for their helpful comments and suggestions that have improved the quality of this paper.

\bibliographystyle{abbrv}
\bibliography{main}

\end{document}